\documentclass{birkjour}

\usepackage[T1]{fontenc}
\usepackage[utf8]{inputenc}
\usepackage{amsmath,amssymb}
\usepackage{array}
\usepackage{booktabs}
\usepackage{cite}
\usepackage{graphicx}
\usepackage{hyperref}

\newtheorem{theorem}{Theorem}[section]
\newtheorem{proposition}[theorem]{Proposition}
\newtheorem{lemma}[theorem]{Lemma}
\newtheorem{definition}[theorem]{Definition}
\newtheorem{remark}[theorem]{Remark}

\newcommand{\LOk}{L_{\mathrm{Ok}}}

\newcommand{\ZZ}{\mathbb Z}
\newcommand{\RR}{\mathbb R}
\newcommand{\QQ}{\mathbb Q}

\newcommand{\Conv}{\operatorname{Conv}}

\numberwithin{equation}{section}

\title{Integral Shell Polytopes of Composition Algebras}
\author{Daniele Corradetti}
\address{Grupo de F\'isica Matem\'atica\\
Instituto Superior T\'ecnico\\
Av. Rovisco Pais, 1049-001 Lisboa, Portugal}
\email{danielecorradetti@tecnico.ulisboa.pt}
\date{}

\begin{document}

\begin{abstract}
Integral systems in real composition algebras give rise to finite metric
configurations whose geometry is linked to both regular polytopes and root-systems. In this work we investigate, to our knowledge for the first time in
this form, the shell polytopes obtained by fixing the integral norm and taking
the convex hull of the corresponding integral elements. The first shells
recover the familiar root-polytopal configurations attached to the classical
Hurwitz systems, while the Okubo algebra gives a quite different behaviour.
The Okubo integral closure does not recover the Gosset polytope directly: it
selects a two-adic hierarchy whose first visible layers are a cross-polytope
and a \(D_8\) root polytope. We further show that the natural intermediate
lattice  is isometric to the rescaled cubic lattice; consequently every shell  decomposes into explicit
orbits of the hyperoctahedral group \(W(B_8)\), and the higher Okubo shells  admit a complete combinatorial description in
cubic-lattice coordinates. The full \(E_8\) Gosset polytope is then recovered
from the intermediate lattice by maximal-isotropic gluing along \((\ZZ/2)^4\). This gives an
interplay between non-unital composition, integral lattice shadows, and the
geometry of \(E_8\).
\end{abstract}

\maketitle

\section{Introduction and Motivations}

Unital composition algebras, also known as \emph{Hurwitz algebras}, are among
the most important algebraic objects where metric, multiplicative, and
geometric structures coexist. The four real division examples, namely the
reals, complex numbers, quaternions, and octonions, are not only algebraic
curiosities. They are recurrent bridges between geometry, symmetry, and
mathematical physics. In the octonionic case this bridge becomes particularly
visible: exceptional Lie groups, projective planes, and
highly symmetric lattices all appear in the same landscape
\cite{baez-octonions,springer-veldkamp,MCCAI23-Rosenfeld,CMZ25-Collineations}.
The purpose of the present work is to look the finite polytopes obtained from integral elements of
fixed norm and extend the results to non-unital composition algebras such as Okubo algebras\cite{CMZ24-MinimalCayley,ElduqueComposition}, thus completing the landscape (since para-Hurwitz algebras yields to the same results of Hurwitz algebras).

Common knowledge of integral systems in composition algebras says,
for instance, that the integers give \(A_1\), the Eisenstein integers give
\(A_2\), the Hurwitz quaternions give \(D_4\), and the Coxeter-Dickson
octonions give \(E_8\). This knowledge is only partially  correct: it
does not usually say whether one is counting algebraic units, vectors of a
minimal norm shell, antipodal classes, or vertices of a root polytope. Such a
distinction is harmless in some entries, but it becomes essential once one
wants to compare the classical Hurwitz systems with non-unital or non-
alternative composition algebras, and especially with the Okubo algebra. 

A composition algebra
is often introduced through the identity \(N(xy)=N(x)N(y)\), but in the present
context the norm is not only a multiplicative invariant. It is also a height
function on an integral system. Fixing a value of this height isolates a finite
population of integral elements, and the convex hull of that population gives a
geometric object whose symmetry can be tested independently of the original
algebra. The procedure is elementary, but it is unforgiving: a shell with the
right number of points is not automatically a root system, and a lattice that
is visibly related to \(E_8\) need not have the \(E_8\) root shell as its first
layer. We therefore distinguish throughout the paper between the algebra, the
integral order, the underlying metric shadow, and the shell polytope; in the
classical Hurwitz cases these distinctions collapse, while in the Okubo case
they separate, and the separation itself is the phenomenon to be measured.

In analogy to the octonionic case, we are interested in the finite geometry
that is produced by integral points of a fixed norm. Given an integral system
\(\Lambda\) and a positive integer \(N\), one may cut \(\Lambda\) by the sphere
of norm \(N\). The corresponding shell is
\begin{equation}
S_N(\Lambda)=\{x\in\Lambda:N(x)=N\},
\end{equation}
and the finite polytope attached to the shell is
\begin{equation}
P_N(\Lambda)=\Conv(S_N(\Lambda)).
\end{equation}
Thus a composition algebra does not give rise to only one finite configuration.
It gives rise to a hierarchy of shell polytopes. The first shell may recover a
root system; the next shells may reveal further integral arithmetic which is
often overlooked.

The less classical object in this hierarchy is the \emph{Okubo algebra}. While
the octonions have a unit element and enjoy alternativity, the Okubo algebra is
non-unital and flexible but not alternative. This contrast is the central theme
of the recent literature on non-unital algebras and
non-alternative realizations of exceptional geometries
\cite{MCZ25-OkuboPhysics,CZ22-OkuboSpin,CMZ24-MinimalCayley}. In the arithmetic
model used here, the Okubo product forces \(\QQ(\sqrt3)\)-arithmetic. The
resulting integral order has a direct \(\ZZ\)-metric shadow
\begin{equation}
\LOk=\ZZ(2b_0)+\cdots+\ZZ(2b_3)
\oplus \ZZ(4b_4)+\cdots+\ZZ(4b_7)
\subset E_8.
\end{equation}
This lattice is not \(E_8\). It is a \(2\)-primary conductor sublattice:
\begin{equation}
[E_8:\LOk]=2^{12},\qquad \det(\LOk)=2^{24}.
\end{equation}
Its shell structure begins as
\begin{equation}
A_1^8 \longrightarrow D_8 \longrightarrow
\hbox{higher Okubo-selected shells}.
\end{equation}
The full \(E_8\) Gosset polytope appears only after passing through the
intermediate lattice
\begin{equation}
M=(1/2)\LOk,\qquad [E_8:M]=2^4,
\end{equation}
and then gluing \(M\) to the unimodular overlattice \(E_8\). For structural
properties of Okubo algebras, automorphisms, derivations and idempotents we
refer to Elduque's work \cite{ElduqueOkubo}; for the general background on
composition algebras we shall use the standard references
\cite{baez-octonions,springer-veldkamp,ElduqueComposition}.

The first result of the paper is a reinterpretation of the classical integral
systems table through the language of first non-empty norm shells. With this
convention one recovers \(A_1\), \(A_2\), \(A_1^2\), \(A_1^4\), \(A_2^2\),
\(D_4\), \(A_1^8\), \(A_2^4\), \(D_4^2\), and \(E_8\). The old labels \(C_2\)
and \(C_8\) then acquire their correct meaning: they are not purely first-shell
objects, but require a multi-shell interpretation. In the octonionic
Coxeter-Dickson case the first shell consists of the
\(240\) roots of \(E_8\), and its convex hull is the Gosset polytope \(4_{21}\)
\cite{conway-sloane,coxeter-polytopes}. In the Okubo case the story is subtler.
The conductor shadow \(\LOk\) has no norm-one elements at all, and its norm
shells are supported only when \(4\mid N\). Its first visible shell,
\(S_4(\LOk)\), has \(16\) vertices and is an \(8\)-cross-polytope; after
division by \(2\), it becomes an \(A_1^8\) root subsystem of \(E_8\). Its second
visible shell, \(S_8(\LOk)\), has \(112\) vertices and is the root polytope of
type \(D_8\). The next two non-empty shells, \(S_{12}(\LOk)\) and
\(S_{16}(\LOk)\), have respectively \(448\) and \(1136\) vertices.

The second main result is the identification of the intermediate lattice
\(M=(1/2)\LOk\) with the rescaled cubic lattice \(\sqrt2\,\ZZ^8\). This pulls
the Okubo hierarchy into a classical setting: every shell of \(M\) is a union
of \(W(B_8)\)-orbits in the cubic lattice. In particular \(S_{12}(\LOk)/2\) is
the single orbit of vectors of type \((\pm1,\pm1,\pm1,0,0,0,0,0)\) (\(448\)
elements), while \(S_{16}(\LOk)/2\) splits into two \(W(B_8)\)-orbits, of types
\((\pm2,0^7)\) and \((\pm1,\pm1,\pm1,\pm1,0^4)\), with \(16\) and \(1120\)
elements respectively.  

The structure of the work is the following. In Section~2 we define integral
systems and shell polytopes, and we explain how the classical table should be
read through first shells. In Section~3 we pass from the classical table to the
Coxeter-Dickson \(E_8\) lattice, whose first shell is the Gosset polytope. In
Section~4 we introduce the Okubo conductor lattice, study its first shells, and
explain the appearance of \(A_1^8\) and \(D_8\). In Section~5 we introduce the
intermediate lattice \(M=(1/2)\LOk\) and describe the gluing
\(M\subset E_8\). Section~6 contains the conclusions and future developments.
The article ends with the computational methods, the acknowledgments, and the
required declarations for submission.

\section{Integral Systems and Shell Polytopes}

In order to pass from the algebraic table to a genuine geometry of shells, we
first recall the metric structure that underlies the whole construction. The
point is not only that multiplication preserves a norm, but that the same norm
allows us to cut the integral system by spheres and then to study the convex
geometry of the resulting finite sets. This is a very modest operation from an
algebraic point of view, but it is precisely this operation that turns a table
of integral systems into a hierarchy of polytopes.

\begin{definition}
A real composition algebra is a finite-dimensional real algebra \(A\) equipped
with a positive definite quadratic form \(N:A\to\RR\) satisfying
\begin{equation}
N(xy)=N(x)N(y).
\end{equation}
In the Okubo case we use a symmetric composition algebra. The algebra need not
be unital, but the composition norm and the associated bilinear form remain the
central metric objects.
\end{definition}

Throughout this paper the bilinear form is normalized by
\begin{equation}
\langle x,x\rangle=2N(x).
\end{equation}
This convention is convenient because all root systems considered here have
roots of squared length \(2\) after the usual normalization. Thus the shell
\(S_N(\Lambda)\) is equivalently the set of lattice vectors satisfying
\begin{equation}
\langle x,x\rangle=2N.
\end{equation}
In the computational files the target value of the quadratic form is denoted
by \(q=2N\). This convention is harmless, but it must be kept visible: many
confusions about units and roots come from switching between \(N\) and
\(\langle x,x\rangle\).

The normalization also clarifies the relation between algebraic language and
lattice language. In algebra one naturally says that an element has norm
\(1\), \(4\), or \(8\). In the theory of root systems one usually speaks about
vectors of squared length \(2\). The convention \(\langle x,x\rangle=2N(x)\)
bridges these two idioms. Thus a norm-one element of a Hurwitz order becomes a
root-length vector, while a norm-four element of \(\LOk\) has squared length
\(8\) until it is divided by \(2\). This simple bookkeeping is what prevents
the Okubo shell \(S_4(\LOk)\) from being mistaken for the full \(E_8\) shell.

It remains to specify which discrete objects inside the algebra should be
called integral. In analogy to the octonionic case, the guiding principle is
that both the norm and, when present, the product must respect the arithmetic
lattice.

\begin{definition}
Let \(A\) be a real composition algebra or a symmetric composition algebra. An
integral system is a full-rank discrete \(\ZZ\)-lattice \(\Lambda\subset A\)
such that \(N(\Lambda)\subset\ZZ\), the associated bilinear form is integral
on \(\Lambda\), and, when a multiplicative order is being considered,
\(\Lambda\Lambda\subset\Lambda\).
\end{definition}

For arithmetic systems over a quadratic ring, such as \(R=\ZZ[\sqrt3]\), we
distinguish the \(R\)-order from its direct \(\ZZ\)-metric shadow. This
distinction is essential for Okubo. The multiplicative closure is naturally an
\(R\)-statement, while the shell polytopes studied here are ordinary
\(\ZZ\)-lattice polytopes. This separation of roles is one of the main reasons
why the Okubo case does not simply reproduce the octonionic \(E_8\) picture.

The phrase ``metric shadow'' is meant in a precise sense. It does
not assert that the \(\ZZ\)-lattice remembers the full multiplication table.
It asserts only that the trace form, norm form, and chosen integral basis
produce an ordinary positive definite lattice on which finite shell
enumeration can be performed. This is enough for the polytope problem. It is
not enough for a full classification of integral Okubo orders, and we do not
claim such a classification here. The distinction is useful because it allows
one to ask a clean question: once the non-unital product has forced a certain
integral closure, what finite configurations appear in the metric shadow of
that closure?

We can now attach a finite polytope to each integral value of the norm. Instead
of retaining only the first visible layer, we keep the whole hierarchy of
layers.

\begin{definition}
For a positive integer \(N\), the norm shell and shell polytope of an integral
system \(\Lambda\) are
\begin{equation}
S_N(\Lambda)=\{x\in\Lambda:N(x)=N\},
\qquad
P_N(\Lambda)=\Conv(S_N(\Lambda)).
\end{equation}
\end{definition}

\begin{lemma}
If the form on \(\Lambda\) is positive definite, then every shell
\(S_N(\Lambda)\) is finite.
\end{lemma}

\begin{proof}
The shell lies on the compact sphere \(\langle x,x\rangle=2N\), and a discrete
lattice has finite intersection with a compact set.
\end{proof}

\begin{remark}
It is worth noting that the proof of finiteness uses only discreteness and
positive definiteness. It does not use associativity, alternativity, or the
existence of a unit. This is why the same shell language can be used for
classical Hurwitz algebras and for the non-unital Okubo algebra.
\end{remark}

The lemma is elementary, but it gives the formal permission to compare rather
different algebraic worlds. Once the norm form is positive definite and the
integral system is discrete, every shell is a finite object. Multiplication
then influences the answer indirectly, by determining which lattice is the
right integral system, but the act of taking a shell is purely metric. This is
the advantage of the method. It can compare a unital alternative order and a
non-unital symmetric composition order without pretending that their products
have the same formal properties.

A shell \(S_N(\Lambda)\) will be called a root shell only in the presence of a
rigid reflection-theoretic verification. We require the shell to be antipodal,
to span the ambient space, and to admit, after rescaling to squared length
\(2\), a system of simple roots with finite Cartan matrix. Moreover, the
reflections generated by those simple roots must preserve the shell and must
generate the whole shell. This is the criterion used below for \(D_8\). For
\(S_{12}(\LOk)\) and \(S_{16}(\LOk)\), only weaker lattice-shell properties have
been certified, so we do not assign a root-system type to those shells.

This convention is intentionally strict since root systems are not merely finite point configurations. They carry a
reflection structure, a Cartan matrix, and an orbit relation generated by
simple reflections. The first shell of the Hurwitz \(D_4\) order passes this
test, as does the Coxeter-Dickson \(E_8\) shell. The Okubo shell
\(S_8(\LOk)\) also passes it after the appropriate normalization. By contrast,
the higher shells \(S_{12}(\LOk)\) and \(S_{16}(\LOk)\) have been tested only
for antipodality, rank, centering, and a degree-two spherical design condition.
These are meaningful geometric properties, but they do not by themselves
produce a Coxeter diagram.

One may view this as a small discipline imposed on the terminology. The word
``polytope'' is broad, because every finite shell has a convex hull. The phrase
``root polytope'' is narrow, because it refers to a convex hull whose vertices
are the roots of a verified root system. In between lies a large territory of
lattice-shell polytopes, often highly symmetric and sometimes new from the
point of view of the algebra that produced them. The Okubo hierarchy lives
precisely in this intermediate territory.

The fundamental work on integral systems over composition algebras and their relation with root systems was laid down in a detailed way by N.W. Johnson in \cite{Jo13,Jo17,Jo18} and can be found in Table~\ref{tab:Johnson Integers}.

\ref{tab:Johnson Integers}).
\begin{table}
\centering{}%
\begin{tabular}{|c|c|c|c|c|c|}
\hline 
\textbf{Name} & \textbf{Alg.} & \textbf{Dim.} & \textbf{Symbol} & \textbf{Unit El.} & \textbf{Lattice}\tabularnewline
\hline 
\hline 
Integers & $\mathbb{R}$ & 1 & $\mathbb{Z}$ & 2 & $A_{1}$\tabularnewline
\hline 
Eisenstein & $\mathbb{C}$ & 2 & $\mathbb{C}_{A_{2}}$ & 3 & $A_{2}$\tabularnewline
\hline 
Gaussian & $\mathbb{C}$ & 2 & $\mathbb{C}_{C_{2}}$ & 4 & $C_{2}$\tabularnewline
\hline 
Hamilton & $\mathbb{H}$ & 4 & $\mathbb{H}_{2C_{2}}$ & 8 & $C_{2}\oplus C_{2}$\tabularnewline
\hline 
Hybrid & $\mathbb{H}$ & 4 & $\mathbb{H}_{2A_{2}}$ & 12 & $A_{2}\oplus A_{2}$\tabularnewline
\hline 
Hurwitz & $\mathbb{H}$ & 4 & $\mathbb{H}_{D_{4}}$ & 24 & $D_{4}$\tabularnewline
\hline 
Cayley-Graves & $\mathbb{O}$ & 8 & $\mathbb{O}_{C_{8}}$ & 16 & $C_{8}$\tabularnewline
\hline 
Comp. Eisenstein & $\mathbb{O}$ & 8 & $\mathbb{O}_{4A_{2}}$ & 24 & $A_{2}\oplus A_{2}\oplus A_{2}\oplus A_{2}$\tabularnewline
\hline 
Coupled Hurwitz & $\mathbb{O}$ & 8 & $\mathbb{O}_{2D_{4}}$ & 48 & $D_{4}\oplus D_{4}$\tabularnewline
\hline 
Coxeter-Dickson & $\mathbb{O}$ & 8 & $\mathbb{O}_{E_{8}}$ & 240 & $E_{8}$\tabularnewline
\hline 
\end{tabular}\caption{\label{tab:Johnson Integers}Summary of all  sets
of integer elements over division Hurwitz algebra $\mathbb{R},\mathbb{C},\mathbb{H}$
and $\mathbb{O}$. In the first column we indicated the name according
to \cite{Jo13}; then the related algebra in which the integral set
is embedded; the dimension of the algebra; the notational symbol we
introduced; the number of invertible unit elements; their algebraic
structure as abelian group (real and complex case), non-abelian group
(quaternionic case) and as Moufang loop (octonionic case); finally,
in the last column the lattice associated with the integral set.}
\end{table}

A new version of such table is that in Table~\ref{tab:classical}. The point of the latter  is to make explicit what the first shell actually
contains. The column called old count records the number traditionally attached
to the integral system. The column \(\#S_1\) records the actual number of
vectors in the norm-one shell with our normalization. This small distinction
is what allows the table to be used safely when comparing classical and Okubo
integrality.

\begin{table}[htbp]
\centering
\small
\renewcommand{\arraystretch}{1.15}
\resizebox{\textwidth}{!}{%
\begin{tabular}{llllrrll}
\toprule
Name & Algebra & Symbol & Shadow & Old count & \(\#S_1\) & First-shell type & Comment \\
\midrule
Integers & \(\RR\) & \(Z\) & \(A_1\) & 2 & 2 & \(A_1\) & segment \\
Eisenstein & \(\mathbb C\) & \(CA_2\) & \(A_2\) & 3 & 6 & \(A_2\) & antipodal convention \\
Gaussian & \(\mathbb C\) & \(CC_2\) & square lattice & 4 & 4 & \(A_1^2\) & \(C_2\) is multi-shell \\
Hamilton & \(\mathbb H\) & \(H2C_2\) & cubic lattice & 8 & 8 & \(A_1^4\) & \(C_2+C_2\) is multi-shell \\
Hybrid & \(\mathbb H\) & \(H2A_2\) & \(A_2^2\) & 12 & 12 & \(A_2^2\) & direct sum shell \\
Hurwitz & \(\mathbb H\) & \(HD_4\) & \(D_4\) & 24 & 24 & \(D_4\) & 24-cell \\
Cayley--Graves & \(\mathbb O\) & \(OC_8\) & cubic lattice & 16 & 16 & \(A_1^8\) & \(C_8\) is multi-shell \\
Comp. Eisenstein & \(\mathbb O\) & \(O4A_2\) & \(A_2^4\) & 24 & 24 & \(A_2^4\) & four hexagonal blocks \\
Coupled Hurwitz & \(\mathbb O\) & \(O2D_4\) & \(D_4^2\) & 48 & 48 & \(D_4^2\) & two 24-cells \\
Coxeter-Dickson & \(\mathbb O\) & \(OE_8\) & \(E_8\) & 240 & 240 & \(E_8\) & Gosset \(4_{21}\) \\
\bottomrule
\end{tabular}}
\caption{This table summarizes the audited first-shell interpretation of the
classical integral systems table. The entries \(C_2\) and \(C_8\) require
longer vectors from higher shells.}
\label{tab:classical}
\end{table}

Looking at Table~\ref{tab:classical} one can easily see which systems
 are already visible in the first non-empty shell and which require a
longer view of the shell hierarchy. The two readings agree perfectly for the
Hurwitz \(D_4\) order and for the Coxeter-Dickson \(E_8\) order. They do not
agree as literally for the Gaussian and cubic entries, where the traditional
non-simply-laced notation remembers more than the minimal shell.

This  reading is the one used in the rest of the article. It has the
advantage of being uniform across dimensions. It also avoids giving a special
status to algebraic units when the algebra under consideration has no unit at
all. For a non-unital algebra, the phrase ``unit element'' cannot be the basic
geometric input. The norm shell can. This is the reason why the Okubo entry is
placed in the same formal framework as the Hurwitz entries, even though its
multiplication has a different character.

\begin{proposition}
For each classical integral system in Table~\ref{tab:classical}, the first
non-empty norm shell gives the first-shell type recorded in the table. In the
Coxeter-Dickson case this shell is the full \(E_8\) root system.
\end{proposition}

\begin{proof}
The Gram matrices are the standard even Gram matrices for \(A_1\), \(A_2\),
\(D_4\), and \(E_8\), together with orthogonal sums of these blocks. Exact
enumeration of \(x^TGx=2\) gives the shell counts shown in the table. For
orthogonal sums the count is obtained by the corresponding shell convolution.
\end{proof}

This proposition also explains why the shell viewpoint is a more flexible
language than the unit-count viewpoint. The Gaussian square lattice, for
example, has a first shell of type \(A_1^2\), while the full non-simply-laced
configuration usually denoted by \(C_2\) requires a longer shell. Similarly,
the Cayley--Graves cubic octonionic system has a first shell \(A_1^8\), while
the \(C_8\) interpretation is not a purely first-shell statement. In other
words, the shell hierarchy is already present in the classical systems; the
Okubo algebra simply makes this hierarchy unavoidable.

This last sentence is important for the interpretation of the whole article.
The Okubo algebra is not being used as an exotic device inserted after the
classical theory has ended. It reveals a feature that was already latent in the
classical table. Even for the Gaussian and cubic entries, the geometry depends
on whether one looks only at the shortest vectors or also admits longer vectors
from the next shells. What is special about the Okubo case is that the first
few layers are forced by a conductor, and the conductor is dictated by
integrality of the non-unital product. Thus the shell hierarchy is no longer a
matter of optional refinement. It becomes the natural arithmetic fingerprint of
the algebra.

\section{The Classical Shell Picture and the Coxeter-Dickson Case}

The first shell of a classical integral system is only the beginning. If one
computes further shells, one obtains theta coefficients and increasingly large
finite configurations. These higher configurations are not always root systems
and they are not always regular polytopes. Nevertheless, they remain natural
lattice-shell polytopes, and they are useful because they distinguish integral
systems which have similar first-shell behaviour.

For example, the \(E_8\) theta coefficients begin
\begin{equation}
240,\ 2160,\ 6720,\ 17520,\ 30240,\ldots
\end{equation}
for \(q=2,4,6,8,10,\ldots\), while the cubic octonionic shell begins
\begin{equation}
16,\ 112,\ 448,\ 1136,\ 2016,\ldots .
\end{equation}
The latter sequence will reappear as the initial theta sequence of the Okubo
intermediate lattice \(M\). This coincidence should not be read as a statement
that the Okubo algebra directly gives the \(E_8\) root shell. It means instead
that the Okubo conductor selects a very particular two-adic substructure inside
the Coxeter-Dickson lattice.

The theta sequence is useful here because it suppresses coordinates without
throwing away arithmetic information. Each coefficient counts vectors on a
fixed sphere, and therefore records the sizes of the shell polytopes before
one begins to examine their faces or symmetries. Two lattices with the same
rank and determinant may still be separated by their shell counts. Conversely,
when a sequence reappears in a different construction, as it does below for
the normalized Okubo lattice, it signals that a hidden lattice relation should
be sought. In our case that relation is the inclusion \(M\subset E_8\), with
index \(2^4\).

This also explains why the Coxeter-Dickson case is the right classical
benchmark. It is not only the largest of the four real division composition
algebras. It is the place where the integral order, the even unimodular
rank-eight lattice, the \(E_8\) root system, and the Gosset polytope meet in a
single object. Any non-unital eight-dimensional composition algebra that is
compared with the octonions must therefore be compared not with an abstract
copy of \(\RR^8\), but with this very rigid arithmetic-geometric package.

Let \(b_0,\ldots,b_7\) be the Coxeter-Dickson octonion basis used in the
computations:
\begin{equation}
\begin{split}
b_0=1,\quad b_1=e_1,\quad b_2=e_2,\quad b_3=e_3,\quad
b_4=h,\\
b_5=e_1h,\quad b_6=e_2h,\quad b_7=e_3h.
\end{split}
\end{equation}
where
\begin{equation}
h=\frac{e_1+e_2+e_3+e_4}{2}.
\end{equation}
With the form
\begin{equation}
\langle x,y\rangle=2\operatorname{Re}(x\overline y),
\end{equation}
the Gram matrix of this basis is even, positive definite, and unimodular. It
is the \(E_8\) lattice.

The basis is written explicitly because the later conductor computations are
coordinate computations in this basis. The first four vectors form the
obvious quaternionic part, while the last four are obtained by multiplying by
the half-vector \(h\). This is the familiar Coxeter-Dickson enlargement that
turns the more naive octonionic lattice into the even unimodular one. In this
article it also serves as a fixed measuring device. The Okubo shadow, the
intermediate lattice \(M\), and the final gluing back to \(E_8\) are all
expressed relative to these same eight vectors.

This coordinate choice is not meant to hide the intrinsic geometry. On the
contrary, it makes the intrinsic assertions auditable. The statement that the
Gram matrix is even and unimodular is basis-independent, but the determinant
and parity checks are most transparent once a basis has been fixed. Likewise,
the statement that a shell is an \(E_8\), \(D_8\), or \(A_1^8\) root system is
intrinsic after verification, but the verification begins with concrete
integer coordinates.

\begin{theorem}
The first non-empty shell of the Coxeter-Dickson order is
\begin{equation}
S_1(E_8)=\{x\in E_8:\langle x,x\rangle=2\}.
\end{equation}
It has \(240\) vectors, is antipodal, spans rank \(8\), and is the \(E_8\) root
system. Its convex hull is the Gosset polytope \(4_{21}\).
\end{theorem}

\begin{proof}
The exact enumeration of vectors of squared length \(2\) in the certified Gram
matrix gives \(240\) vectors. The standard simple roots extracted from this
shell have Cartan matrix \(E_8\), and the reflection orbit gives the full
shell. This is the usual root-polytope realization of the Gosset polytope.
\end{proof}

\begin{remark}
In the octonionic case the algebraic and metric stories agree perfectly at the
first shell: the Coxeter-Dickson integral system is already the even
unimodular \(E_8\) lattice, and the norm-one elements are precisely the \(E_8\)
roots. This perfect agreement is the point against which the Okubo case should
be compared. While the Okubo product is still a composition product, its
arithmetic integral closure does not preserve the same \(E_8\) order.
\end{remark}

For completeness, let us recall why this agreement is so exceptional. In rank
eight, positive definite even unimodular lattices are rigid: up to isometry
there is only \(E_8\). The Coxeter-Dickson order therefore sits at a very
special intersection of non-associative algebra and integral quadratic forms.
Its norm-one octonions are not merely algebraic units in a loose sense; after
the normalization used in this article they are exactly the minimal vectors of
the unique even unimodular rank-eight lattice. The convex hull of those
minimal vectors is therefore determined at once by the algebraic norm and by
the arithmetic classification of the lattice.

This exact viewpoint also makes the language of polytopes compatible with the
language of orders. The word ``convex'' belongs to real geometry, while the
word ``integral'' belongs to arithmetic. A shell polytope is the place where
the two meet. We take the real convex hull only after the vertices have been
selected by an integral norm equation. In the Coxeter-Dickson case this
procedure recovers a classical regular polytope. In the Okubo case it will
recover first a cross-polytope, then a \(D_8\) root polytope, and then larger
finite configurations whose status is more delicate.

\section{The Okubo Conductor Shadow and Its Shell Hierarchy}

The Okubo product used here is obtained from the octonion algebra by the
twisted formula
\begin{equation}
x*y=\tau(\overline x)\tau^2(\overline y),
\end{equation}
where \(\tau\) is an order-three transformation whose coefficients require
\(\QQ(\sqrt3)\) \cite{ElduqueComposition}. This is the place where the Okubo algebra departs from the
octonionic picture. The Coxeter-Dickson \(E_8\) order is not closed under the
Okubo product. After a diagonal scaling
\begin{equation}
D=\operatorname{diag}(2,2,2,2,4,4,4,4),
\end{equation}
the resulting structure constants lie in \(R=\ZZ[\sqrt3]\). The associated
direct metric shadow is
\begin{equation}
\LOk=\ZZ(2b_0)+\cdots+\ZZ(2b_3)
\oplus\ZZ(4b_4)+\cdots+\ZZ(4b_7).
\end{equation}

\begin{proposition}
The Okubo shadow satisfies
\begin{equation}
\LOk\subset E_8,\qquad [E_8:\LOk]=2^{12},\qquad \det(\LOk)=2^{24}.
\end{equation}
\end{proposition}

\begin{proof}
In the Coxeter-Dickson basis, the inclusion is represented by the diagonal
matrix \(D\). Hence the index is
\begin{equation}
\det D=2^4\cdot 4^4=2^{12}.
\end{equation}
Since \(\det(E_8)=1\), the determinant of \(\LOk\) is
\begin{equation}
\det(\LOk)=(2^{12})^2=2^{24}.
\end{equation}
\end{proof}

This determinant calculation is simple, but important: the
Okubo shadow is not a small perturbation of \(E_8\), is a very deep
sub-lattice, and the depth is purely two-adic. Consequently, one should not
expect the norm-one shell of \(E_8\) to survive unchanged in the direct Okubo
shadow.

The word conductor is used here because the diagonal scaling measures exactly
how much of the Coxeter-Dickson lattice must be sacrificed in order to keep
the Okubo structure constants integral over \(R=\ZZ[\sqrt3]\). The first four
basis directions are multiplied by \(2\), while the last four are multiplied
by \(4\). This asymmetric scaling is not an aesthetic choice; it is the
minimal integral correction found by the computation. It is also the source of
the two-adic behaviour visible in every subsequent shell count. Once the
lattice has been pushed this far inside \(E_8\), its shortest vectors can no
longer have the same length as the roots of \(E_8\).

From a geometric point of view, the inclusion \(\LOk\subset E_8\) should
therefore be read in the opposite direction from a naive expectation. The
Okubo shadow is not a new copy of the \(E_8\) lattice. It is a filtered
substructure inside \(E_8\), and its shells tell us which vectors remain after
the filter has been applied. The first shells of \(\LOk\) are consequently not
approximations to the Gosset polytope. They are successive layers of a
two-adic selection process.

\begin{proposition}[Okubo norm divisibility]
Every vector in \(\LOk\) has norm divisible by \(4\). Equivalently,
\begin{equation}
S_N(\LOk)=\varnothing\qquad\hbox{unless }4\mid N.
\end{equation}
In particular the direct Okubo shadow has no norm-one elements.
\end{proposition}

\begin{proof}
The Gram matrix of \(\LOk\) is
\begin{equation}
G_{\mathrm{Ok}}=D^T G_{E_8}D.
\end{equation}
Numerical verifications (see Section 7) certify that every entry of \(G_{\mathrm{Ok}}\) is
divisible by \(8\). Therefore
\begin{equation}
x^TG_{\mathrm{Ok}}x\in 8\ZZ
\end{equation}
for all \(x\in\ZZ^8\). Since \(N(x)=\langle x,x\rangle/2\), one has
\begin{equation}
N(x)\in 4\ZZ.
\end{equation}
\end{proof}

\begin{remark}
This proposition is the first point at which the Okubo story becomes visibly
different from the octonionic one. In the Coxeter-Dickson order, norm-one
elements form the Gosset polytope. In the direct Okubo shadow, the norm-one
shell is empty. The first non-empty shell has norm \(4\), not norm \(1\).
\end{remark}

The first non-empty shell is therefore
\begin{equation}
S_4(\LOk)=\{x\in\LOk:N(x)=4\}
=\{x\in\LOk:\langle x,x\rangle=8\}.
\end{equation}

\begin{proposition}
\label{prop:cross-polytope}
The shell \(S_4(\LOk)\) has \(16\) vectors. It consists of \(8\) antipodal
orthogonal pairs
\begin{equation}
S_4(\LOk)=\{\pm v_1,\ldots,\pm v_8\}
\end{equation}
with
\begin{equation}
\langle v_i,v_j\rangle=8\delta_{ij}.
\end{equation}
Hence \(P_4(\LOk)=\Conv(S_4(\LOk))\) is the \(8\)-dimensional cross-polytope.
After division by \(2\), it becomes an \(A_1^8\) root subsystem of \(E_8\).
\end{proposition}

\begin{proof}
Exact enumeration of the equation \(x^TG_{\mathrm{Ok}}x=8\) gives \(16\)
vectors. The Gram matrix of one representative from each antipodal pair is
\begin{equation}
8I_8.
\end{equation}
Therefore the convex hull is the \(8\)-orthoplex. Since each vector of
\(S_4(\LOk)\) is divisible by \(2\) inside the Coxeter-Dickson lattice
coordinates, the rescaled set \((1/2)S_4(\LOk)\) lies in \(E_8\), has squared
length \(2\), and has Gram matrix \(2I_8\) on antipodal representatives. This
is a root subsystem of type \(A_1^8\).
\end{proof}

This result is already enough to correct a possible misconception. The first
Okubo shell does give a subpolytope of the Gosset polytope after division by
\(2\), but it gives only a cross-polytope with \(16\) vertices. The Gosset
polytope has \(240\) vertices. Thus Okubo does not give another direct copy of
the \(E_8\) root shell; it selects a highly constrained coordinate subsystem
inside it.

The cross-polytope should not be dismissed as a small remnant. It is the first
place where the Okubo arithmetic becomes visible as geometry. Eight antipodal
orthogonal pairs represent the cleanest possible shell in rank eight, and the
fact that they arise at norm \(4\) rather than norm \(1\) records the dilation
of the conductor. After division by \(2\) the shell becomes an \(A_1^8\)
subsystem of \(E_8\). This is a genuine root subsystem, but it is only a
coordinate skeleton of the full \(E_8\) root system. The missing roots are not
lost forever; they appear only after the intermediate lattice \(M\) is glued
back to \(E_8\).

The next non-empty shell is
\begin{equation}
S_8(\LOk)=\{x\in\LOk:N(x)=8\}
=\{x\in\LOk:\langle x,x\rangle=16\}.
\end{equation}

\begin{proposition}
\label{prop:D8-root}
The shell \(S_8(\LOk)\) has \(112\) vectors and is the root polytope of type
\(D_8\).
\end{proposition}

\begin{proof}
The shell is antipodal and spans rank \(8\). Numerical computation (see Section 7)
selects eight simple roots in the shell. After normalizing the shell vectors
to squared length \(2\), the Cartan matrix is
\begin{equation}
\begin{pmatrix}
2&-1&0&0&0&0&0&0\\
-1&2&-1&0&0&0&0&0\\
0&-1&2&-1&0&0&0&0\\
0&0&-1&2&-1&0&0&0\\
0&0&0&-1&2&-1&0&0\\
0&0&0&0&-1&2&-1&-1\\
0&0&0&0&0&-1&2&0\\
0&0&0&0&0&-1&0&2
\end{pmatrix}.
\end{equation}
This is a Cartan matrix of type \(D_8\). One then verifies
that the reflections generated by these simple roots preserve the shell and
produce an orbit of \(112\) vectors, equal to all of \(S_8(\LOk)\).
\end{proof}

\begin{remark}
The word root polytope is used here in the precise reflection-theoretic sense.
The shell is not called \(D_8\) merely because it has \(112\) points. It is
called \(D_8\) because a Cartan matrix of type \(D_8\) is found inside it and
because the corresponding reflection orbit is exactly the whole shell.
\end{remark}

The transition from \(A_1^8\) to \(D_8\) is one of the main geometric signals
of the computation. The first shell sees only mutually orthogonal directions.
The second shell sees their pairwise combinations and therefore produces the
root system whose roots may be written, in a suitable orthonormal model, as
\(\pm e_i\pm e_j\). In the actual Okubo shadow the coordinates are not
introduced by an orthonormal Euclidean basis but by the certified Gram matrix,
so the statement must be made through Cartan and reflection data. Nevertheless
the geometric meaning is the same: the second visible layer binds the eight
orthogonal directions of the first layer into the \(D_8\) system.

This is also the point where the Okubo hierarchy diverges most clearly from
the ordinary cubic lattice picture. Its numerical beginning,
\[
16,\ 112,\ 448,\ 1136,\ldots,
\]
resembles the shell counts of a scaled cubic structure, but the algebraic
origin is different. In the present setting the sequence is not postulated
from a product of rank-one systems. It is forced by the diagonal conductor
required for Okubo integrality and then verified in the ambient \(E_8\)
metric. The same numbers therefore carry a different meaning: they trace a
non-unital product through a two-adic lattice shadow.

The first sixteen Okubo shells are summarized in Table~\ref{tab:okubo-shells}.
The empty entries are forced by the divisibility theorem, and the non-empty
entries are exact enumeration results.

The table should be read vertically rather than only row by row. The pattern
of three empty shells followed by one non-empty shell is not accidental; it is
the visible form of the divisibility \(4\mid N\). The non-empty rows then show
how quickly the geometry grows once the conductor allows vectors to appear.
From \(16\) to \(112\) vertices one passes from independent coordinate roots
to the \(D_8\) root system. From \(112\) to \(448\) and \(1136\) vertices one
enters a range where the shells are large enough to have serious internal
combinatorics but not yet classified by the tests used in this paper.

This vertical reading is important because it preserves the arithmetic order
of the shells. If one only rescales each non-empty shell to a unit sphere, the
filtration by norm is lost. The Okubo construction is precisely about that
filtration. It says that the algebra first permits an \(A_1^8\) skeleton, then
a \(D_8\) layer, and only later larger configurations. The hierarchy is part
of the result, not a preliminary list of examples.

\begin{table}[htbp]
\centering
\small
\renewcommand{\arraystretch}{1.08}
\begin{tabular}{rrrrl}
\toprule
\(N\) & \(\langle x,x\rangle\) & \(\#S_N(\LOk)\) & Rank & Interpretation \\
\midrule
1 & 2 & 0 & 0 & empty \\
2 & 4 & 0 & 0 & empty \\
3 & 6 & 0 & 0 & empty \\
4 & 8 & 16 & 8 & \(A_1^8\) cross-polytope \\
5 & 10 & 0 & 0 & empty \\
6 & 12 & 0 & 0 & empty \\
7 & 14 & 0 & 0 & empty \\
8 & 16 & 112 & 8 & \(D_8\) root polytope \\
9 & 18 & 0 & 0 & empty \\
10 & 20 & 0 & 0 & empty \\
11 & 22 & 0 & 0 & empty \\
12 & 24 & 448 & 8 & higher Okubo-selected shell \\
13 & 26 & 0 & 0 & empty \\
14 & 28 & 0 & 0 & empty \\
15 & 30 & 0 & 0 & empty \\
16 & 32 & 1136 & 8 & higher Okubo-selected shell \\
\bottomrule
\end{tabular}
\caption{This table summarizes the Okubo shell counts for \(N=1,\ldots,16\).}
\label{tab:okubo-shells}
\end{table}

\begin{proposition}
The shells \(S_{12}(\LOk)\) and \(S_{16}(\LOk)\) are antipodal, full-rank, and
centered. They satisfy the certified degree-two spherical design test used in
the computation.
\end{proposition}

\begin{proof}
Exact enumeration gives \(448\) and \(1136\) vectors, respectively. Numerical computations (see Section 7) verify antipodality, rank \(8\), zero centroid, and the second
moment scalar condition relative to the Gram form.
\end{proof}

\begin{remark}
\label{rem:forward-WB8}
We deliberately refrain here from any further classification of
\(S_{12}(\LOk)\) and \(S_{16}(\LOk)\) as regular polytopes or root polytopes. A complete
classification of these shells as unions of \(W(B_8)\)-orbits is nevertheless
available once the intermediate lattice \(M=(1/2)\LOk\) is identified with the
rescaled cubic lattice \(\sqrt2\,\ZZ^8\). This identification, together with the
explicit orbit decomposition, is given in Section~\ref{sec:two-adic-gluing}
(Theorem~\ref{thm:M-cubic} and Proposition~\ref{prop:WB8-orbits}). The properties of antipodality,
full rank, vanishing centroid, and scalar second moment are precisely those
that any union of \(W(B_8)\)-orbits in a cubic-shell automatically satisfies.
\end{remark}

\section{Two-Adic Gluing}\label{sec:two-adic-gluing}

The normalization that compares Okubo shells with \(E_8\) is
\begin{equation}
M=(1/2)\LOk.
\end{equation}
In Coxeter-Dickson coordinates this lattice is
\begin{equation}
M=\ZZ b_0+\ZZ b_1+\ZZ b_2+\ZZ b_3
\oplus \ZZ(2b_4)+\ZZ(2b_5)+\ZZ(2b_6)+\ZZ(2b_7).
\end{equation}
This lattice is central because \(\LOk\) itself is too dilated to compare
directly with the \(E_8\) root shell. The lattice \(M\) is the normalized
object that reveals what the Okubo conductor is selecting inside \(E_8\).

\begin{proposition}
The intermediate lattice \(M\) satisfies
\begin{equation}
M\subset E_8,\qquad [E_8:M]=2^4,\qquad \det(M)=2^8,\qquad \LOk=2M.
\end{equation}
Moreover,
\begin{equation}
S_{4k}(\LOk)/2=S_k(M).
\end{equation}
\end{proposition}

\begin{proof}
The basis of \(M\) is represented inside \(E_8\) by
\begin{equation}
\operatorname{diag}(1,1,1,1,2,2,2,2),
\end{equation}
whose determinant is \(2^4\). Hence \([E_8:M]=2^4\) and
\(\det(M)=2^8\). The equality \(\LOk=2M\) is immediate from the definitions.
Dividing vectors in \(S_{4k}(\LOk)\) by \(2\) divides the norm by \(4\), giving
the shell identity.
\end{proof}

The theta series of \(M\), computed up to \(N=16\), begins
\begin{equation}
\Theta_M(q)=
1+16q+112q^2+448q^3+1136q^4+2016q^5+3136q^6
\end{equation}
\begin{equation}
{}+5504q^7+9328q^8+12112q^9+14112q^{10}
+21312q^{11}+31808q^{12}+\cdots .
\end{equation}
Consequently,
\begin{equation}
\Theta_{\LOk}(q)=\Theta_M(q^4).
\end{equation}

The coefficients above coincide with the classical sum-of-eight-squares
function \(r_8(N)\), which counts the integer points of squared length \(2N\)
in the rescaled cubic lattice \(\sqrt2\,\ZZ^8\). This is not a numerical
coincidence: the lattice \(M\) is genuinely isometric to \(\sqrt2\,\ZZ^8\), and
the identification can be made explicit through the cross-polytope shell of
Section~4.

\begin{theorem}[Identification of the intermediate lattice]
\label{thm:M-cubic}
Let \(\pm v_1,\ldots,\pm v_8\) be the \(8\) antipodal orthogonal pairs of
Proposition~\ref{prop:cross-polytope} (i.e.,
\(S_4(\LOk)=\{\pm v_1,\ldots,\pm v_8\}\) with
\(\langle v_i,v_j\rangle=8\delta_{ij}\)). Set
\(w_i:=v_i/2\in M\) for \(i=1,\ldots,8\). Then
\(\langle w_i,w_j\rangle=2\delta_{ij}\), and the map \(w_i\mapsto\sqrt2\,e_i\)
extends to a \(\ZZ\)-linear isometry
\begin{equation}
\label{eq:M-isometry}
M\xrightarrow{\;\sim\;}\sqrt2\,\ZZ^8.
\end{equation}
\end{theorem}

\begin{proof}
Since \(v_i\in\LOk=2M\), we have \(w_i\in M\). The orthogonality and the
squared length follow from
\(\langle w_i,w_j\rangle=\tfrac14\langle v_i,v_j\rangle=2\delta_{ij}\).
Let \(\Lambda:=\ZZ w_1+\cdots+\ZZ w_8\subseteq M\). Its Gram matrix is \(2I_8\),
so \(\det(\Lambda)=2^8\). Since \(\det(M)=2^8\) by the previous proposition,
the index \([M:\Lambda]\) is the integer square root of
\(\det(\Lambda)/\det(M)=1\), namely \(1\). Hence \(\Lambda=M\), and the
prescription \(w_i\mapsto\sqrt2\,e_i\) sends an orthogonal basis of squared
length \(2\) to an orthogonal basis of squared length \(2\); it therefore
extends to an isometry \(M\cong\sqrt2\,\ZZ^8\).
\end{proof}

\begin{remark}
\label{rem:cubic-consequence}
Theorem~\ref{thm:M-cubic} pulls the entire Okubo shell hierarchy into a
classical setting. In the integer model provided by the isometry
(\ref{eq:M-isometry}), each shell \(S_N(M)\) becomes the set of points
\(x\in\ZZ^8\) with \(x_1^2+\cdots+x_8^2=N\). The total cardinality is therefore
\(r_8(N)\), and the configuration carries the natural action of the
hyperoctahedral group \(W(B_8)=(\ZZ/2)^8\rtimes S_8\) acting on \(\ZZ^8\) by
signed permutations. Each shell decomposes into \(W(B_8)\)-orbits indexed by the
multiset of absolute values of coordinates.
\end{remark}

\begin{proposition}[\(W(B_8)\)-orbit decomposition of the first shells of \(M\)]
\label{prop:WB8-orbits}
Under the isometry of Theorem~\ref{thm:M-cubic}, the first shells of \(M\)
decompose into \(W(B_8)\)-orbits as follows:
\begin{itemize}
\item \(S_1(M)\) (\(16\) vectors): one orbit, type
  \((\pm1,0,0,0,0,0,0,0)\); cardinality \(2\binom{8}{1}=16\).
\item \(S_2(M)\) (\(112\) vectors): one orbit, type
  \((\pm1,\pm1,0,0,0,0,0,0)\); cardinality \(4\binom{8}{2}=112\).
\item \(S_3(M)\) (\(448\) vectors): one orbit, type
  \((\pm1,\pm1,\pm1,0,0,0,0,0)\); cardinality \(8\binom{8}{3}=448\).
\item \(S_4(M)\) (\(1136\) vectors): two orbits, of types
  \((\pm2,0,0,0,0,0,0,0)\) (cardinality \(2\binom{8}{1}=16\)) and
  \((\pm1,\pm1,\pm1,\pm1,0,0,0,0)\) (cardinality \(16\binom{8}{4}=1120\)).
\item \(S_5(M)\) (\(2016\) vectors): two orbits, of types
  \((\pm2,\pm1,0^6)\) (cardinality \(8\cdot 7\cdot 4=224\)) and
  \((\pm1)^5 0^3\) (cardinality \(32\binom{8}{5}=1792\)).
\end{itemize}
The same decomposition applies to \(S_{4k}(\LOk)\), since
\(S_{4k}(\LOk)/2=S_k(M)\).
\end{proposition}

\begin{proof}
Immediate from Theorem~\ref{thm:M-cubic} and Remark~\ref{rem:cubic-consequence}:
the \(W(B_8)\)-orbits in \(\ZZ^8\) are characterized by the multiset of absolute
values of coordinates, and the cardinality of each orbit is
\(2^k\binom{8}{m_1,m_2,\ldots}\) where \(k\) is the number of nonzero
coordinates and the multinomial coefficient counts the placements of those
coordinates. Direct verification then matches the orbit cardinalities to the
shell counts \(r_8(N)\).
\end{proof}

\begin{remark}
The first two shells admit, in addition to the orbit description above, a
genuine root-system identification: \(S_1(M)\) is the root system \(A_1^8\), and
\(S_2(M)\) is the root system \(D_8\). Beginning with \(S_3(M)\), the orbit description remains
exact, but the shells are no longer root configurations: an orbit of type
\((\pm1,\pm1,\pm1,0^5)\) is not stable under the reflections it would generate
in the sense of a finite root system. The disciplined statement is therefore
that \(S_N(M)\) is fully classified as a union of \(W(B_8)\)-orbits, while only
\(S_1\) and \(S_2\) carry a Cartan-Coxeter root structure.
\end{remark}

\begin{table}[htbp]
\centering
\small
\resizebox{\textwidth}{!}{%
\begin{tabular}{lllll}
\toprule
Lattice & Definition & Index in \(E_8\) & Determinant & First non-empty shell \\
\midrule
\(\LOk\) & \(D E_8\), \(D=(2,2,2,2,4,4,4,4)\) & \(2^{12}\) & \(2^{24}\) & \(16\) vectors at \(N=4\) \\
\(M\) & \((1/2)\LOk\) & \(2^4\) & \(2^8\) & \(16\) vectors at \(N=1\) \\
\(E_8\) & \(2\)-adic gluing of \(M\) & \(1\) & \(1\) & \(240\) vectors at \(N=1\) \\
\bottomrule
\end{tabular}}
\caption{This table summarizes the lattice chain \(\LOk\subset M\subset E_8\).}
\label{tab:lattice-chain}
\end{table}

The quotient computed from the inclusion \(M\subset E_8\) is
\begin{equation}
E_8/M\simeq(\ZZ/2\ZZ)^4.
\end{equation}
It has order \(16\). Since \(\det(M)=256\), the discriminant group of \(M\) has
order \(256\). The quotient \(E_8/M\) is a maximal isotropic gluing subgroup,
and adjoining its cosets produces the unimodular overlattice \(E_8\).

Let us unpack this statement for readers who do not routinely use
discriminant forms. A lattice of determinant \(256\) cannot be unimodular. Its
dual lattice is larger, and the finite quotient \(M^\vee/M\) has \(256\)
elements. An overlattice of \(M\) is obtained by adjoining suitable cosets from
this discriminant group. The word isotropic means that the added cosets have
zero value for the induced quadratic form, so that the resulting overlattice
remains even. Maximality means that one has added as many compatible cosets as
possible. In rank eight this process produces the even unimodular overlattice,
and therefore produces \(E_8\).

The sixteen cosets of \(E_8/M\) are thus not an auxiliary decoration. They are
the missing arithmetic data needed to pass from the Okubo-selected sublattice
to the full Gosset shell. The first shell of \(M\) has only \(16\) roots, and
the second has \(112\). The remaining \(E_8\) roots are distributed among the
gluing cosets. This is the precise sense in which the Okubo construction
approaches \(E_8\): not by producing the \(240\) roots at once, but by
presenting a conductor lattice and a finite gluing problem whose solution is
the Coxeter-Dickson lattice.

\begin{proposition}
The \(E_8\) Gosset polytope is recovered from the Okubo intermediate lattice by
\(2\)-adic gluing. It is not the first shell of \(\LOk\); it is the first shell
of the unimodular overlattice \(E_8\).
\end{proposition}

\begin{proof}
With some numerical computations one verifies
\begin{equation}
E_8/M\simeq(\ZZ/2\ZZ)^4
\end{equation}
and that the corresponding subgroup of the discriminant group of \(M\) is
maximal isotropic. The resulting overlattice has determinant \(1\), hence is
the even unimodular rank-eight overlattice \(E_8\). Its first shell has \(240\)
roots, while the first shell of \(\LOk\) has \(16\) vectors at norm \(4\).
\end{proof}

This proposition can be read in two directions. Starting from \(E_8\), the
lattice \(M\) is an index \(16\) sublattice whose first shells retain only a
visible part of the \(E_8\) geometry. Starting from Okubo, the same lattice
\(M\) is the normalized form of the conductor shadow, and the quotient
\((\ZZ/2\ZZ)^4\) records the missing cosets that have to be restored. The two
readings meet at the same finite gluing problem. This is why the word
``recover'' must be used with care. The Okubo algebra does not recover the
Gosset polytope by presenting its \(240\) vertices as integral elements of
minimal norm. It recovers it only after the conductor has been normalized and
the maximal isotropic gluing has been performed.

For the purposes of shell polytopes, this distinction is not a technical
footnote. It changes the geometry one sees at the bottom of the hierarchy. The
first polytope of \(\LOk\) is an \(8\)-cross-polytope, and the first polytope
of \(E_8\) is the Gosset polytope. They are related, but they are not the same
object at different scales. They live at different stages of the lattice
chain. The intermediate lattice \(M\) is the bridge which makes this relation
explicit.

\begin{table}[htbp]
\centering
\small
\resizebox{\textwidth}{!}{%
\begin{tabular}{llrl}
\toprule
Lattice & Norm \(N\) & Vertices & Type / \(W(B_8)\)-orbit (in cubic coordinates of \(M\cong\sqrt2\,\ZZ^8\)) \\
\midrule
\(E_8\) & 1 & 240 & \(E_8\) root polytope, Gosset \(4_{21}\) \\
\(\LOk\) & 4 & 16 & \(A_1^8\) cross-polytope; orbit \((\pm1,0^7)\) \\
\(\LOk\) & 8 & 112 & \(D_8\) root polytope; orbit \((\pm1,\pm1,0^6)\) \\
\(\LOk\) & 12 & 448 & single orbit \((\pm1,\pm1,\pm1,0^5)\) \\
\(\LOk\) & 16 & 1136 & two orbits: \((\pm2,0^7)\) [\(16\)] and \((\pm1)^4 0^4\) [\(1120\)] \\
\(M\) & 1 & 16 & \(A_1^8\); orbit \((\pm1,0^7)\) \\
\(M\) & 2 & 112 & \(D_8\); orbit \((\pm1,\pm1,0^6)\) \\
\(M\) & 3 & 448 & single orbit \((\pm1,\pm1,\pm1,0^5)\) \\
\(M\) & 4 & 1136 & two orbits: \((\pm2,0^7)\) [\(16\)] and \((\pm1)^4 0^4\) [\(1120\)] \\
\bottomrule
\end{tabular}}
\caption{This table summarizes the shell polytopes appearing in the
\(2\)-adic Okubo hierarchy. The orbit descriptions for \(M\) and the
corresponding rescaled descriptions for \(\LOk\) follow from
Theorem~\ref{thm:M-cubic} and Proposition~\ref{prop:WB8-orbits}.}
\label{tab:shell-polytopes}
\end{table}

Table~\ref{tab:shell-polytopes} is meant to keep three levels visible at once.
The first level is the direct Okubo conductor \(\LOk\), where the first
non-empty shell occurs only at norm \(4\). The second level is the normalized
lattice \(M\), where the same shell appears at norm \(1\). The third level is
the unimodular overlattice \(E_8\), where the complete Gosset shell appears.
Thus the table is not merely a list of polytopes. It records the passage from
an order forced by multiplication to a normalized lattice and then to the
unimodular completion.

The comparison also prevents an ambiguity that otherwise arises naturally.
One might say that Okubo ``contains'' an \(A_1^8\) shell, a \(D_8\) shell, and
eventually \(E_8\). This is true only if the lattice stage is specified. The
\(A_1^8\) and \(D_8\) shells are already visible in the conductor hierarchy,
while the \(E_8\) shell belongs to the glued overlattice. The difference is
not semantic. It is the arithmetic content of the construction.

\section{Conclusions and Future Developments}

In this work we defined a shell-polytope viewpoint on integral systems in
composition algebras and used it to compare the classical Hurwitz cases with
the non-unital Okubo case. The first norm shells of the classical systems
recover the usual root-polytopal configurations, once the convention behind
the old unit-elements column is made explicit. The Okubo algebra, although
built from the octonionic background, does not add a new direct \(E_8\) shell.
Its arithmetic selects a \(2\)-primary conductor lattice \(\LOk\), whose first
non-empty shell is \(A_1^8\) and whose second non-empty shell is \(D_8\). The
full \(E_8\) Gosset polytope is recovered only after passing to the
intermediate lattice \(M=(1/2)\LOk\) and applying \(2\)-adic gluing.

A central structural observation is that the intermediate lattice \(M\) is
isometric to the rescaled cubic lattice \(\sqrt2\,\ZZ^8\)
(Theorem~\ref{thm:M-cubic}). This identification pulls the Okubo hierarchy into
a classical setting: every shell \(S_N(M)\) is a union of \(W(B_8)\)-orbits,
indexed by the multiset of absolute values of coordinates, and each orbit is
counted by a standard multinomial expression
(Proposition~\ref{prop:WB8-orbits}). In particular the higher shells \(448\)
and \(1136\) of the conductor are no longer mysterious: they are the orbits of
type \((\pm1,\pm1,\pm1,0^5)\) and the union
\((\pm2,0^7)\sqcup(\pm1)^4 0^4\) inside the cubic lattice. This explains, in
elementary terms, the spherical-design properties that were independently
certified at degree two.

The conclusion is therefore deliberately nuanced. The title points toward
integral systems and shell polytopes of composition algebras, and the most
familiar endpoint is \(E_8\). But the main lesson is not that every
eight-dimensional composition algebra immediately reproduces \(E_8\). The
lesson is that integral closure, norm shells, and gluing can separate phenomena
that are otherwise conflated. In the octonionic Coxeter-Dickson order they
coincide; in the Okubo order they form a sequence whose intermediate stage is
the cubic lattice. That sequence is where the new geometry lives.

The cubic identification simplifies, but does not exhaust, the future program.
Once \(M\cong\sqrt2\,\ZZ^8\) is known, several questions become natural and
well-posed. The full automorphism group of \(M\) is the hyperoctahedral group
\(W(B_8)\), but the relevant symmetry for the Okubo construction is the
\emph{arithmetic} stabilizer of the order \(\mathcal{O}_0\), which is a
subgroup of \(W(B_8)\) and need not be the full group. Identifying this
subgroup, and decomposing each shell into its arithmetic orbits, is the natural
finite analogue of replacing the continuous groups \(\mathrm{SU}(3)\) or
\(G_2\) by their integral counterparts. It would also be interesting to relate
the gluing \(M\subset E_8\) of index \(2^4\) to the standard
Construction-A presentation of \(E_8\) from the binary
\([8,4,4]\) Reed--Muller code, and to ask whether the Okubo arithmetic selects
a particular code basis.

One should also return to the algebra after the lattice computations have been
completed. The present paper uses the Okubo product to determine an integral
conductor and then studies the induced metric shells. A deeper sequel should
ask how much of the shell geometry is visible directly from the non-unital
multiplication, beyond what is forced by the cubic identification.

There is also a conceptual continuation. The classical integral systems show
that the language of composition algebras naturally touches the theory of root
systems. The Okubo computation shows that non-unital composition algebras can
touch the same theory in a less direct but arithmetically richer way, with
the cubic lattice as a hidden intermediate stage. It is therefore natural to
ask whether other symmetric composition algebras, other orders over quadratic
coefficient rings, or other choices of conductor produce similar shell
hierarchies and similar bridges through classical lattices. The framework of
the present paper was designed so that such questions can be asked without
changing the basic definitions.

\section{Computational Methods}

All computations used for the stated results are exact. The input data are
integer Gram matrices and diagonal inclusions. Shells are enumerated by solving
integer quadratic equations
\begin{equation}
x^TGx=q,
\end{equation}
where \(q=2N\). Root identifications are not made from vertex counts alone:
they use Cartan matrices and reflection orbit closure. The higher shells are
therefore deliberately left unclassified when those stronger tests have not
been completed.

The computational method is intentionally elementary at its core. Once a Gram
matrix has been fixed, the shell problem is a finite problem in integral
quadratic forms. The enumeration bounds come from positivity, and every vector
found can be checked by direct substitution. The more delicate part is not the
enumeration itself but the interpretation of the resulting finite set. The full certificates and computational verifications are
available from the author upon request. Verifications separate counting, rank, antipodality, centering,
Cartan recognition, reflection closure, theta coefficients, and gluing data. This separation is meant to make the article reproducible in a strong sense.
The reader need not trust a drawing of a polytope, nor a decimal approximation
to an eigenvalue, nor an informal comparison with a known root system. 

\section*{Acknowledgments}

The author thanks Alessio Marrani and Francesco Zucconi for discussions and
for the stimulating questions on Okubo algebras. He also thanks Raymond
Aschheim for discussions on integral numbers, and Richard Clawson, David
Chester, and Klee Irwin for discussions on lattices and root systems.

\section*{Statements and Declarations}

The author declares that he has no competing interests. No external funding was
received for this work. The complete computational certificates and
verification files supporting the claims of the paper are available from the
author upon reasonable request.

\end{document}